\documentclass{article}
\usepackage[a4paper, top=1.25in, bottom=1.25in, left=1.25in, right=1.25in]{geometry}
\usepackage[utf8]{inputenc}

\usepackage[font=footnotesize,labelfont=bf]{caption}
\usepackage{url}
\usepackage{amsmath,amsfonts,amssymb,amsthm}
\usepackage{enumitem}
\usepackage{mathtools}
\usepackage{braket}
\usepackage{algorithm,algpseudocode}
\usepackage{siunitx}[=v2]
\usepackage{bm}
\usepackage{bbm}
\usepackage{mleftright}
\usepackage{booktabs}
\usepackage[dvipsnames]{xcolor}
\usepackage{pgfplots}
\usepackage{tikz}
\usepackage{color}
\usepackage{hyperref}
\hypersetup{colorlinks=true,
    linkcolor=BrickRed,
    citecolor=OliveGreen,
    urlcolor=MidnightBlue}
\usepackage{enumitem}
\usepackage{cleveref}

\usepackage{subcaption}
\usepackage[export]{adjustbox}

\usepackage{booktabs}
\usepackage{graphicx}
\usepackage{ragged2e}

\pgfplotsset{compat=1.9}
\usepgfplotslibrary{fillbetween}

\DeclareMathOperator*{\subjto}{subj. to}

\DeclareMathOperator{\tr}{tr}
\DeclareMathOperator{\epigraph}{epi}
\DeclareMathOperator{\hypograph}{hypo}
\DeclareMathOperator{\closure}{cl}

\DeclareMathOperator{\domain}{dom}
\DeclareMathOperator{\diag}{diag}
\DeclareMathOperator{\image}{im}

\DeclareMathOperator{\im}{Im}

\DeclarePairedDelimiterX{\divx}[2]{(}{)}{#1\mspace{1.5mu}\delimsize\|\mspace{1.5mu}#2}
\DeclarePairedDelimiterX{\divy}[2]{(}{)}{#1\mspace{1mu}\delimsize|\mspace{1mu}#2}
\DeclarePairedDelimiterX{\divz}[3]{(}{)}{#1; #2\mspace{1.5mu}\delimsize\|\mspace{1.5mu}#3}
\DeclarePairedDelimiterX{\inp}[2]{\langle}{\rangle}{#1, #2}
\DeclarePairedDelimiterX{\norm}[1]{\lVert}{\rVert}{#1}
\DeclarePairedDelimiterX{\abs}[1]{\lvert}{\rvert}{#1}
\DeclarePairedDelimiterX{\bk}[2]{\langle}{\rangle}{#1 \delimsize\vert #2}

\newcommand*\wc{{\mkern 2mu\cdot\mkern 2mu}}
\newcommand*\conj[1]{\bar{#1}}

\NewDocumentCommand{\grad}{e{_^}}{%
  \mathop{}\!
  \nabla
  \IfValueT{#1}{_{\!#1}}
  \IfValueT{#2}{^{#2}}
}

\NewCommandCopy{\ordinaryforall}{\forall}
\NewCommandCopy{\ordinaryexists}{\exists}

\RenewDocumentCommand{\forall}{}{\mathop{{}\ordinaryforall}}
\RenewDocumentCommand{\exists}{}{\mathop{{}\ordinaryexists}}

\usepackage{listings}
\definecolor{codegreen}{rgb}{0,0.6,0}
\definecolor{codegray}{rgb}{0.5,0.5,0.5}
\definecolor{outputgray}{rgb}{0.35,0.35,0.35}
\definecolor{codepurple}{rgb}{0.58,0,0.82}
\definecolor{backcolour}{rgb}{0.95,0.95,0.92}
\lstdefinestyle{mystyle}{
    backgroundcolor=\color{backcolour},   
    commentstyle=\color{codegreen},
    keywordstyle=\color{magenta},
    numberstyle=\tiny\color{codegray},
    stringstyle=\color{codepurple},
    basicstyle=\ttfamily\footnotesize,
    breakatwhitespace=false,         
    breaklines=true,                 
    captionpos=t,                    
    keepspaces=true,                 
    numbers=left,                    
    numbersep=5pt,                  
    showspaces=false,                
    showstringspaces=false,
    showtabs=false,                  
    tabsize=2,
    frame=lrtb,
    xleftmargin=0.75em,
    framexleftmargin=1.5em,
    framesep=0pt,
    framexleftmargin=0.25em
}
\newcommand{\bftab}{\fontseries{b}\selectfont}

\newtheorem{thm}{Theorem}[section]
\newtheorem{lem}[thm]{Lemma}
\newtheorem{prop}[thm]{Proposition}
\newtheorem{cor}[thm]{Corollary}
\newtheorem{defn}[thm]{Definition}
\newtheorem{rem}[thm]{Remark}

\newtheorem{conjecture}[thm]{Conjecture}

\newcommand{\footremember}[2]{%
    \footnote{#2}
    \newcounter{#1}
    \setcounter{#1}{\value{footnote}}%
}
\newcommand{\footrecall}[1]{%
    \footnotemark[\value{#1}]%
} 

\begin{document}

\title{Operator convexity along lines, self-concordance, and sandwiched R\'enyi entropies}



\author{%
    Kerry He\footremember{monash}{Department of Electrical and Computer Systems Engineering, Monash University, Clayton VIC 3800, Australia. \url{{kerry.he1, james.saunderson}@monash.edu}} \and James Saunderson\footrecall{monash} \and Hamza Fawzi\footremember{cambridge}{Department of Applied Mathematics and Theoretical Physics, University of Cambridge, Cambridge CB3 0WA, United Kingdom. \url{h.fawzi@damtp.cam.ac.uk}}
}
\date{}

\newcommand{\HF}[1]{{\color{magenta}[HF: #1]}}

\maketitle

\begin{abstract}
    Barrier methods play a central role in the theory and practice of convex optimization. One of the most general and successful analyses of barrier methods for convex optimization, due to Nesterov and Nemirovskii, relies on the notion of self-concordance. While an extremely powerful concept, proving self-concordance of barrier functions can be very difficult. In this paper we give a simple way to verify that the natural logarithmic barrier of a convex nonlinear constraint is self-concordant via the theory of operator convex functions. Namely, we show that if a convex function is operator convex along any one-dimensional restriction, then the natural logarithmic barrier of its epigraph is self-concordant. We apply this technique to construct self-concordant barriers for the epigraphs of functions arising in quantum information theory. Notably, we apply this to the sandwiched R\'enyi entropy function, for which no self-concordant barrier was known before. Additionally, we utilize our sufficient condition to provide simplified proofs for previously established self-concordance results for the noncommutative perspective of operator convex functions. An implementation of the convex cones considered in this paper is now available in our open source interior-point solver \href{https://github.com/kerry-he/qics}{QICS}.
\end{abstract}

\section{Introduction}\label{sec:intro}


Let $f:\mathbb{R}^n\rightarrow\mathbb{R}$ be a convex function. Convex optimization problems involving this function can often be expressed in terms of its epigraph, i.e.,
\begin{equation*}
    \epigraph f \coloneqq \{ (t,x) \in \mathbb{R} \times \domain f : t \geq f(x) \}.
\end{equation*}
If we have a self-concordant barrier for this set, then this allows us to incorporate the function $f$ into the broader Nesterov-Nemirovskii framework for interior-point methods~\cite{nesterov1994interior}.
However, constructing an efficiently computable self-concordant barrier for a set is not always straightforward. 
If a function $G$ is a self-concordant barrier for the domain of $f$, then it would be natural to hope that a self-concordant barrier for the epigraph of $f$ is
\begin{equation}\label{eqn:candidate-barrier}
    (t,x) \mapsto -\log(t - f(x)) + G(x).
\end{equation}
However, this is not true in general, e.g., when $f(x)=e^x$ (see~\cite[Proposition 5.3.3]{nesterov1994interior} for an actual self-concordant barrier for the epigraph of the exponential).
Moreover, even in cases where~\eqref{eqn:candidate-barrier} is self-concordant, proving this can be difficult. 
For example, functions arising in quantum information theory, such as the sandwiched R\'enyi entropy of the title, are often spectral functions of multiple Hermitian matrices. Although it is possible to obtain explicit expressions for the derivatives of these functions (see, e.g., Section~\ref{sec:derivatives}), they typically depend, in a complicated way, on the eigendecompositions of the matrices involved. This makes it challenging
to prove self-concordance of~\eqref{eqn:candidate-barrier}, which requires uniformly bounding the third derivative in terms of the second derivative.




In this paper, instead of proving self-concordance by directly working with the derivatives of the barrier function, we instead relate self-concordance to operator convexity, i.e., univariate functions that are convex with respect to the Loewner order when extended to spectral functions of Hermitian matrices
(see Section~\ref{subsec:operator} for a precise definition). In particular, we show that if a function, when restricted to any line within its domain, is operator convex, then the natural logarithmic barrier~\eqref{eqn:candidate-barrier} is self-concordant. This allows us to prove self-concordance by instead using tools from the rich literature of operator convex functions. 
We summarize this idea in the following theorem, which is a simplification of the main technical result of our paper which we present later in Theorem~\ref{thm:compatibility-operator-concave}.

\begin{thm}\label{thm:simplified}
    Let $\mathbb{V}$ be a finite-dimensional real vector space, and $f:\domain f\rightarrow\mathbb{R}$ be a $C^3$ function with open domain $\domain f \subset \mathbb{V}$. Suppose that for all $x\in\domain f$ and $h\in\mathbb{V}$ such that $x\pm h \in \closure\domain f$, the function
    \begin{equation*}
        t \mapsto f(x + th),
    \end{equation*}
    is operator convex on $(-1, 1)$. If $G$ is a $\nu$-self-concordant barrier for $\closure \domain f$, then 
    \begin{equation*}
        (t,x) \mapsto -\log(t - f(x)) + G(x),
    \end{equation*}
    defined on $\mathbb{R}\times\domain f$ is a $(1+\nu)$-self-concordant barrier for $\closure\epigraph f$.
\end{thm}
\begin{proof}
    See Section~\ref{sec:proof-main}.
\end{proof}

We prove this theorem by showing that if a function is operator convex along lines, then it satisfies another previously known sufficient condition for~\eqref{eqn:candidate-barrier} to be a self-concordant barrier, i.e., it is compatible with respect to its domain, in the sense of Nesterov and Nemirovskii.
We provide a more detailed background on this concept in Section~\ref{sec:self-concordant}. In previous works~\cite{faybusovich2017matrix,coey2023conic,fawzi2023optimal}, it was shown that operator convex functions and their noncommutative perspectives were compatible with their domains. Our key assumption of being operator convex along lines is a weaker condition which not only generalizes the results from these works, but also allows us to prove compatibility of more complicated expressions.




Specifically, we focus on constructing self-concordant barriers for epigraphs (and hypographs) of functions called sandwiched R\'enyi (quasi-relative) entropies. These functions have remarkable convexity and concavity properties, and play a prominent role in quantum information theory. However, until now they have not been amenable to optimization via off-the-shelf interior point methods. We discuss these functions in more detail in the following section.

\subsection{Sandwiched R\'enyi entropies}\label{sec:renyi}

Consider the following trace function, sometimes referred to as the sandwiched $\alpha$-quasi-relative entropy,
\begin{equation}\label{eqn:trace-function}
    \Psi_\alpha(X, Y) \coloneqq \tr\!\left[ \left(Y^\frac{1-\alpha}{2\alpha} X Y^\frac{1-\alpha}{2\alpha} \right)^\alpha \right],
\end{equation}
defined on $\mathbb{H}^n_{++}\times\mathbb{H}^n_{++}$, i.e., pairs of positive definite $n\times n$ Hermitian matrices. This function is used to define the sandwiched $\alpha$-R\'enyi entropy~\cite{muller2013quantum,wilde2014strong}
\begin{equation*}
    D_\alpha\divx{X}{Y} \coloneqq \frac{1}{\alpha - 1} \log(\Psi_\alpha(X, Y)),
\end{equation*}
where $\alpha\in(0, 1)\cup(1, \infty)$. These sandwiched R\'enyi entropies are used to quantify how dissimilar two quantum states are, and arise in applications such as quantum hypothesis testing~\cite{mosonyi2015quantum} and quantum cryptography~\cite{dupuis2023privacy}. 


\paragraph{Convexity properties}

The sandwiched quasi-relative entropy $\Psi_\alpha$ is jointly concave for $\alpha\in[\frac{1}{2}, 1]$, and is jointly convex for $\alpha\in[1, \infty)$, see, e.g.,~\cite{frank2013monotonicity}.
Various techniques have been developed to prove such convexity and concavity results for $\Psi_\alpha$ and related trace functions. We comment on three notable techniques. First, in the works which originally introduced the sandwiched R\'enyi entropy~\cite{muller2013quantum,wilde2014strong}, convexity of $\Psi_\alpha$ for $\alpha\in[1,2]$ was shown by expressing the function as an appropriate composition between noncommutative perspectives of operator convex functions and positive linear maps. Second, a complex analysis technique based on Epstein's method~\cite{epstein1973remarks} was used by Hiai~\cite{hiai2001concavity,hiai2013concavity,hiai2016concavity} to prove concavity of a more general class of trace functions. Notably, these results are a generalization of concavity of $\Psi_\alpha$ for $\alpha\in[\frac{1}{2}, 1]$. Third, a variational technique was used in~\cite{frank2013monotonicity} (see, also,~\cite{zhang2020wigner}), to prove concavity and convexity of $\Psi_\alpha$ for the full range $\alpha\in[\frac{1}{2}, \infty)$. 

Given these results, joint convexity of the sandwiched R\'enyi entropy $D_\alpha$ for $\alpha\in[\frac{1}{2}, 1)$ follows from a simple composition argument, see, e.g.,~\cite[Section 3.2.4]{boyd2004convex}. It is also easy to show that $D_\alpha$ is neither concave nor convex for $\alpha\in(1, \infty)$ by noticing that, in the scalar case, $D_\alpha\divx{x}{y}=\alpha/(\alpha-1)\log(x)-\log(y)$ for $x,y>0$, i.e., $D_\alpha$ is concave in $x$ and convex in $y$. Instead, if we wish to minimize the sandwiched R\'enyi entropy over any set $\mathcal{F}\subseteq\domain\Psi_\alpha$, due to monotonicity of $x\mapsto\log(x)$ we recognize that
\begin{equation}\label{eqn:reformulate}
    \min_{(X,Y)\in\mathcal{F}} D_\alpha\divx{X}{Y} = \begin{cases}
        \displaystyle \frac{1}{\alpha-1} \log \biggl(\max_{(X,Y)\in\mathcal{F}} \Psi_\alpha(X, Y) \biggr), \quad &\text{if }\alpha\in[1/2, 1)\\[10pt]
        \displaystyle \frac{1}{\alpha-1} \log \biggl(\min_{(X,Y)\in\mathcal{F}} \Psi_\alpha(X, Y) \biggr), \quad &\text{if }\alpha\in(1, \infty).
    \end{cases}
\end{equation}
Notably, the optimization problems in the right-hand expressions are both convex, and therefore it suffices to develop efficient optimization techniques to minimize or maximize $\Psi_\alpha$ for appropriate corresponding ranges of $\alpha$.

\paragraph{Optimizing R\'enyi entropies}

Currently, there is a lack of efficient optimization techniques available to minimize the sandwiched R\'enyi entropy. A first-order method known as entropic mirror descent was proposed in~\cite{you2022minimizing} to minimize these functions. However the algorithm is not guaranteed to converge to the optimal solution, only to a neighborhood around it. 

For some choices of $\alpha$, there are relatively well-known techniques to minimize the sandwiched R\'enyi entropy. When $\alpha=\frac{1}{2}$, the sandwiched quasi-relative entropy $\Psi_{1/2}$ is equal to the square root of the fidelity function $F(X,Y)=\norm{\sqrt{X}\sqrt{Y}}^2_1$ (where $\norm{\wc}_1$ denotes the trace norm), which has a well-known semidefinite programming representation~\cite{watrous2012simpler} given by
\begin{equation*}
    \min_{Z\in\mathbb{C}^{n\times n}} \quad \frac{1}{2} \tr[Z + Z^*] \quad \subjto \quad \begin{bmatrix}
        X & Z \\ Z^* & Y
    \end{bmatrix} \succeq0.
\end{equation*}

When $\alpha\rightarrow1$, the (normalized) sandwiched R\'enyi entropy converges to the (normalized) quantum relative entropy, i.e.,
\begin{equation*}
    \lim_{\alpha\rightarrow1} \bm{D}_\alpha\divz{\tr[X]}{X}{Y} = D_1\divx{X}{Y},
\end{equation*}
where $\bm{D}_\alpha$ is the perspective of the sandwiched R\'enyi entropy (see~\eqref{eqn:perspective-of-entropy} for a precise definition), and 
\begin{equation*}
    D_1\divx{X}{Y} = \tr[X\log(X)] - \tr[X\log(Y)],
\end{equation*}
is the (Umegaki) quantum relative entropy, which is known to be jointly convex~\cite{effros2009matrix}. Similarly, a closely related function to the sandwiched R\'enyi entropy is the R\'enyi entropy~\cite{petz1986quasi}, which is defined as $\hat{D}_\alpha\divx{X}{Y}=\log(\hat\Psi_\alpha(X,Y))/(\alpha-1)$ where
\begin{equation*}
    \hat{\Psi}_\alpha(X, Y) = \tr[X^\alpha Y^{1-\alpha}],
\end{equation*}
is sometimes referred to as the $\alpha$-quasi-relative entropy. Note that $\Psi_\alpha$ and $\hat{\Psi}_\alpha$ agree when their matrix arguments commute. The quasi-relative entropy $\hat{\Psi}_\alpha$ is jointly concave for $\alpha\in[0, 1]$ and jointly convex for $\alpha\in[-1, 0]\cup [1, 2]$, results which directly follow from theorems of Lieb~\cite{lieb1973convex} and Ando~\cite{ando1979concavity}. For the quantum relative entropy $D_1$ and quasi-relative entropy $\hat{\Psi}_\alpha$, it was shown in~\cite{fawzi2023optimal} that natural barriers for the epigraphs or hypographs of these functions are self-concordant, and therefore optimization problems minimizing these functions could be efficiently solved using interior-point methods, see, e.g.,~\cite{he2024exploiting,he2024qics,karimi2023efficient,coey2023performance}. Alternatively, it was shown in~\cite{fawzi2019semidefinite} and~\cite{fawzi2017lieb} that these functions could be approximated using linear matrix inequalities, and could therefore be optimized using semidefinite programming software.




\paragraph{Main results}

Using Theorem~\ref{thm:compatibility-operator-concave} (which we recall is a generalization of Theorem~\ref{thm:simplified}), we show that the natural logarithmic barrier functions for the hypographs of $\Psi_\alpha$ for $\alpha\in[\frac{1}{2}, 1]$ and epigraphs of $\Psi_\alpha$ for $\alpha\in[1, 2]$ are self-concordant with optimal barrier parameter, as summarized below (see Section~\ref{sec:self-concordant} for terminology related to self-concordant barriers).

\begin{thm}\label{thm:renyi-barrier}
    Let $n$ be any positive integer.
    \begin{enumerate}[label=(\roman*), ref=\ref{thm:renyi-barrier}(\roman*)]
        \item \label{thm:renyi-barrier-i} If $\alpha\in[\frac{1}{2}, 1]$, then the function 
        \begin{equation*}
            (t, X, Y)\in\mathbb{R}\times\mathbb{H}^n_{++}\times\mathbb{H}^n_{++} \mapsto -\log(\Psi_\alpha(X, Y) - t) - \log\det(X) - \log\det(Y),
        \end{equation*}
        is a $(1+2n)$-logarithmically homogeneous self-concordant barrier for
        \begin{equation*}
            \closure \hypograph \Psi_\alpha = \closure \{ (t, X, Y)\in\mathbb{R}\times\mathbb{H}^n_{++}\times\mathbb{H}^n_{++} : t \leq \Psi_\alpha(X, Y)  \}. 
        \end{equation*}     
        \item \label{thm:renyi-barrier-ii} If $\alpha\in[1, 2]$, then the function 
        \begin{equation*}
            (t, X, Y)\in\mathbb{R}\times\mathbb{H}^n_{++}\times\mathbb{H}^n_{++} \mapsto -\log(t - \Psi_\alpha(X, Y)) - \log\det(X) - \log\det(Y),
        \end{equation*}
        is a $(1+2n)$-logarithmically homogeneous self-concordant barrier for
        \begin{equation*}
            \closure \epigraph \Psi_\alpha = \closure \{ (t, X, Y)\in\mathbb{R}\times\mathbb{H}^n_{++}\times\mathbb{H}^n_{++} : t \geq \Psi_\alpha(X, Y)  \}. 
        \end{equation*}
    \end{enumerate}
    Moreover, these barriers are optimal in the sense that any self-concordant barrier for $\closure \hypograph \Psi_\alpha$ when $\alpha\in[\frac{1}{2}, 1]$ and $\closure \epigraph \Psi_\alpha$ when $\alpha\in[1, 2]$ has parameter at least $1 + 2n$.
\end{thm}
\begin{proof}
    See Section~\ref{sec:main-proof}.
\end{proof}


Note that $\closure \epigraph \Psi_\alpha$ and $\closure \hypograph \Psi_\alpha$ are both proper convex cones for the appropriate ranges of $\alpha$ as $\Psi_\alpha$ is positively homogeneous of degree one, i.e., $\Psi_\alpha(\lambda X, \lambda Y)=\lambda \Psi_\alpha(X,Y)$ for all $\lambda>0$ and $X,Y\in\mathbb{H}^n_{++}$. Although we know that the epigraph of $\Psi_\alpha$ for $\alpha\in(2,\infty)$ is also a convex cone, we are not aware of an efficiently computable self-concordant barrier for this cone. We provide a brief discussion about this range of $\alpha$ in Section~\ref{sec:conclusion}.

If $\alpha\in[\frac{1}{2}, 1)$, the sandwiched R\'enyi entropy $D_\alpha$ is convex. In this setting, we can directly give a self-concordant barrier for the (conic hull of the) epigraph of $D_\alpha$, i.e., the epigraph of the perspective of the sandwiched R\'enyi entropy
\begin{equation}\label{eqn:perspective-of-entropy}
    \bm{D}_\alpha\divz{u}{X}{Y} \coloneqq u D_\alpha\divx{u^{-1}X}{u^{-1}Y},
\end{equation}
which is defined on $\mathbb{R}_{++}\times\mathbb{H}^n_{++}\times\mathbb{H}^n_{++}$.

\begin{thm}\label{thm:direct-renyi-barrier}
    For any positive integer $n$ and $\alpha\in[\frac{1}{2}, 1)$, the function 
    \begin{equation*}
        (t, u, X, Y)\in\mathbb{R}\times\mathbb{R}_{++}\times\mathbb{H}^n_{++}\times\mathbb{H}^n_{++} \mapsto -\log(t - \bm{D}_\alpha\divz{u}{X}{Y}) - \log(u)-\log\det(X) - \log\det(Y),
    \end{equation*}
    is a $(2+2n)$-self-concordant barrier for
    \begin{equation*}
        \closure \epigraph \bm{D}_\alpha = \closure \{ (t, u, X, Y)\in\mathbb{R}\times\mathbb{R}_{++}\times\mathbb{H}^n_{++}\times\mathbb{H}^n_{++} : t \geq \bm{D}_\alpha\divz{u}{X}{Y} \}. 
    \end{equation*}
    Moreover, this barrier is optimal in the sense that any self-concordant barrier for $\closure\epigraph \bm{D}_\alpha$ when $\alpha\in[\frac{1}{2}, 1)$ has parameter at least $2+2n$. 
\end{thm}
\begin{proof}
    See Section~\ref{sec:main-proof}.
\end{proof}






\section{Preliminaries}

We denote the set of $n \times n$ Hermitian matrices as $\mathbb{H}^n$ with trace inner product $\inp{X}{Y}=\tr[X^*Y]$, where $X^*$ denotes the conjugate transpose of a complex matrix $X$. Similarly, we denote the positive semidefinite cone as $\mathbb{H}^n_+$, and its interior as $\mathbb{H}^n_{++}$. In the remainder of this section, we provide some background on operator monotone and operator convex functions, as well as self-concordant barriers. 


\subsection{Operator monotonicity and convexity}\label{subsec:operator}
Consider a finite-dimensional real vector space $\mathbb{V}$ and proper convex cone $\mathcal{K}\subset\mathbb{V}$. We define the partial ordering $x \succeq_\mathcal{K} y$ to mean $x - y \in \mathcal{K}$ for $x,y\in\mathbb{V}$. When we omit the subscript, we refer to the Loewner ordering, i.e., we use $X \succeq Y$ to mean $X-Y\in\mathbb{H}^n_+$ for $X,Y\in\mathbb{H}^n$. 

Now consider a second finite-dimensional real vector space $\mathbb{V}'$, a proper convex cone $\mathcal{K}'\subset\mathbb{V}'$, and a function $f:\mathbb{V}'\rightarrow\mathbb{V}$. We say that $f$ is $(\mathcal{K}', \mathcal{K})$\emph{-monotone} if for all $x,y\in\domain f$ we have
\begin{equation*}
     x\succeq_{\mathcal{K}'} y \quad \Longrightarrow \quad f(x) \succeq_\mathcal{K} f(y).
\end{equation*}
We say that the function $f$ is $\mathcal{K}$\emph{-convex} if for all $x,y\in\domain f$ we have
\begin{equation*}
     \lambda f(x) + (1-\lambda) f(y) \succeq_\mathcal{K}  f(\lambda x + (1-\lambda)y), \qquad \forall \lambda\in[0,1].
\end{equation*}
Similarly, we say that $f$ is $\mathcal{K}$-concave if $-f$ is $\mathcal{K}$-convex.

Now consider a real-valued function $g$ defined on the interval $(a,b)$ where $-\infty\leq a<b\leq\infty$. We can extend this function to be defined on Hermitian matrices $X\in\mathbb{H}^n$ whose eigenvalues are in $(a,b)$ as follows. If $X$ has the spectral decomposition $X=\sum_{i=1}^n\lambda_i v_iv_i^*$ where $\lambda_i\in(a, b)$ for all $i=1,\ldots,n$, then we define $g(X) = \sum_{i=1}^n g(\lambda_i) v_iv_i^*$. Given this, we say that the function $g$ is \emph{operator monotone} if for all positive integers $n$ and matrices $X,Y\in\mathbb{H}^n$ with eigenvalues in $(a,b)$, we have
\begin{equation*}
    X\succeq Y \quad \Longrightarrow \quad g(X) \succeq g(Y),
\end{equation*}
i.e., $g$ is $(\mathbb{H}^n_+,\mathbb{H}^n_+)$-monotone for all positive integers $n$. We say that $g$ is \emph{operator convex} if for all positive integers $n$ and matrices $X,Y\in\mathbb{H}^n$ with eigenvalues in $(a,b)$, we have
\begin{equation*}
    \lambda g(X) + (1-\lambda) g(Y) \succeq  g(\lambda X + (1-\lambda)Y), \qquad \forall \lambda\in[0,1],
\end{equation*}
i.e., $g$ is $\mathbb{H}^n_+$-convex for all positive integers $n$. Similarly, we say $g$ is operator concave if $-g$ is operator convex.

Important examples of these functions include $x\mapsto\log(x)$, $x\mapsto x^p$ for $p\in[0, 1]$, and $x\mapsto -x^p$ for $p\in[-1, 0]$, which are all operator monotone and operator concave on $(0, \infty)$. Similarly, the function $x\mapsto x^p$ for $p\in(1,2]$ is operator convex on $(0, \infty)$, but is not operator monotone. See, e.g.,~\cite[Theorem 2.6]{carlen2018inequalities}, for a proof of these results.

There is an elegant theory behind operator monotone and operator convex functions (see, e.g.,~\cite{hiai2010matrix,simon2019loewner}). One important result is Loewner's theorem, which relates operator monotone functions to Pick functions, and which provides us with an integral representation for the class of operator monotone functions.

\begin{lem}[Loewner's Theorem]\label{lem:loewner}
    For a real-valued function $g$ defined on the interval $(a, b)$, where $-\infty\leq a < b \leq \infty$, the following statements are equivalent:
    \begin{enumerate}[label=(\roman*), ref=\ref{lem:loewner}(\roman*)]
        \item $g$ is operator monotone on $(a, b)$.
        \item $g$ has an analytic continuation from $(a, b)$ to the upper half-plane $\mathbb{C}^+\coloneqq\{ z\in\mathbb{C} : \im z > 0 \}$ that maps $\mathbb{C}^+$ into $\mathbb{C}^+$ (i.e., $g$ is a \emph{Pick function}). 
        \item $g$ has the following integral representation
        \begin{equation*}
            g(x) = \alpha + \beta x + \int_{\mathbb{R} \setminus (a,b) } \frac{1}{s-x}-\frac{s}{s^2+1}\,d\mu(s), \quad\forall x\in(a, b),
        \end{equation*}
        where $\alpha\in\mathbb{R}$, $\beta\geq0$, and $\mu$ is a positive finite Borel measure on $\mathbb{R}\setminus(a,b)$.
    \end{enumerate}    
\end{lem}

It is also well known that operator convex functions have a similar integral representation. 

\begin{lem}[{\cite[Theorem 2.7.6]{hiai2010matrix}}]\label{lem:operator-convex}
    Let $g$ be an operator convex function defined on the interval $(-1, 1)$. Then there exists a unique positive finite Borel measure $\mu$ on $[-1, 1]$ such that
    \begin{equation*}
        g(x)=g(0) + g'(0)x + \frac{1}{2} g''(0) \int_{-1}^1 \frac{x^2}{1-sx} \, d\mu(s), \quad \forall x\in(-1, 1).
    \end{equation*}
\end{lem}

\subsection{Self-concordant barriers}\label{sec:self-concordant}

For a finite-dimensional real vector space $\mathbb{V}$, consider a $C^3$, closed, strictly convex function $F$ with open domain $\domain F\coloneqq\{ x\in\mathbb{V} : F(x) < \infty \}$. We say that $F$ is \emph{self-concordant} if
\begin{equation*}
    \abs{\mathsf{D}^3F(x)[h, h, h]} \leq 2(\mathsf{D}^2F(x)[h, h])^{3/2},
\end{equation*}
for all $x\in\domain F$ and $h\in\mathbb{V}$, where
\begin{equation*}
    \mathsf{D}^k F(x)[h_1,\ldots,h_k] = \frac{\partial^k}{\partial t_1\cdots\partial t_k} \bigg|_{t_1=\ldots=t_k=0}\ F(x+ t_1h_1 + \cdots + t_kh_k),
\end{equation*}
denotes the $k$-th directional derivative at $x$ along the directions $h_1,\ldots,h_k$. Since $\domain F$ is open and the epigraph of $F$ is closed, $F$ is a barrier function for $\closure\domain F$, as any sequence $x_k\in\domain F$ converging to the boundary of $\domain F$ satisfies $F(x_k)\rightarrow\infty$~\cite[Theorem 5.13]{nesterov2018lectures}. Additionally, $F$ is a $\nu$-self-concordant barrier if
\begin{equation*}
    2 \mathsf{D}F(x)[h] - \mathsf{D}^2F(x)[h,h] \leq \nu,
\end{equation*}
for all $x\in\domain F$ and $h\in\mathbb{V}$. When $\domain F$ is a convex cone, we say that $F$ is $\nu$-logarithmically homogeneous if
\begin{equation*}
    F(tx) = F(x) - \nu\log(t),
\end{equation*}
for all $x\in\domain F$ and $t>0$. If $F$ is a $\nu$-logarithmically homogeneous self-concordant barrier, then it is also a $\nu$-self-concordant barrier~\cite[Lemma 5.4.3]{nesterov2018lectures}. 

Similar to~\cite{fawzi2023optimal}, the main technique we will use to construct logarithmically homogeneous self-concordant barriers is to use compatibility of functions with respect to their domains, which we define below. Note that, by convention, compatibility is defined in terms of concave functions rather than convex functions. Because of this, in what follows we will typically formulate our general results in terms of concave functions and their hypographs.
\begin{defn}[{\cite[Definition 5.1.1]{nesterov1994interior}}]\label{def:compatibility}
    Let $\mathbb{V}$ and $\mathbb{V}'$ be finite-dimensional real vector spaces, let $\mathcal{K}\subset\mathbb{V}'$ be a closed, convex cone, and let $f:\domain f\rightarrow\mathbb{V}'$ be a $\mathcal{K}$-concave $C^3$ function with open domain $\domain f \subset \mathbb{V}$. Then $f$ is $(\mathcal{K}, \beta)$\emph{-compatible} with the domain $\closure\domain f$ if there exists $\beta\geq0$ such that
    \begin{equation*}
        \mathsf{D}^3 f(x)[h, h, h] \preceq_\mathcal{K} -3\beta \mathsf{D}^2 f(x)[h, h],
    \end{equation*}
    for all $x\in\domain f$ and $h\in\mathbb{V}$ such that $x\pm h \in \closure\domain f$.
\end{defn}
Once we have established compatibility of a function, we can use the following result to construct a self-concordant barrier for the hypograph of the function.
\begin{lem}[{\cite[Theorem 5.4.4]{nesterov2018lectures}}]\label{lem:compatibility-to-barrier}
    Let $\mathbb{V}$ and $\mathbb{V}'$ be finite-dimensional real vector spaces, and let $\mathcal{K}\subset\mathbb{V}'$ be a closed, convex cone. Let $f:\domain f \rightarrow \mathbb{V}'$ be a $\mathcal{K}$-concave $C^3$ function with open domain $\domain f \subset \mathbb{V}$, and which is $(\mathcal{K}, \beta)$-compatible with $\closure\domain f$. Let $G$ be a $\nu$-self-concordant barrier for $\closure\domain f$, and $H$ be an $\eta$-self-concordant barrier for $\mathcal{K}$. Then 
    \begin{equation*}
        (t, x) \mapsto H(f(x) - t) + \beta^3G(x),
    \end{equation*}
    defined on the domain $\mathbb{V}' \times \domain f$ is an $(\eta+\beta^3\nu)$-self-concordant barrier for the set
    \begin{equation*}
        \closure\hypograph f = \closure \{ (t, x) \in \mathbb{V}'\times \domain f : t \preceq_\mathcal{K} f(x) \}.
    \end{equation*}
\end{lem}
\begin{proof}
    Using the notation from~\cite[Theorem 5.4.4]{nesterov2018lectures}, let $\xi=f$, $E_1=\mathbb{V}$, $E_2=E_3=\mathbb{V}'$, $Q=\domain f$, $Q_2=\{ (y,z)\in\mathbb{V}'\times\mathbb{V}' : y \succeq_{\mathcal{K}} z \}$, $\Phi(y,z)=H(y-z)$, and $F=G$. Clearly, any element of $\mathcal{K}\times\{0\}$ is a recession direction for $Q_2$, and additionally
    \begin{equation*}
        \closure\hypograph f = \closure \{ (z,x)\in \mathbb{V}'\times\domain f : \exists y\in\mathbb{V}',\ f(x)\succeq_{\mathcal{K}} y,\ (y,z)\in Q_2 \},
    \end{equation*}
    which is enough to recognize that~\cite[Theorem 5.4.4]{nesterov2018lectures} implies our desired result.
\end{proof}

We also introduce two important composition rules for compatibility with linear and affine maps when $\mathcal{K}=\mathbb{H}^n_+$ from~\cite[Proposition 3.4]{fawzi2023optimal} and~\cite[Lemma 5.1.3(iii)]{nesterov1994interior}.
\begin{lem}\label{lem:compatibility-composition}
    Let $\mathbb{V}$ and $\mathbb{V}'$ be finite-dimensional real vector spaces, let $f:\domain f \subset \mathbb{V} \rightarrow \mathbb{H}^n$ be a $\mathbb{H}^n_+$-concave $C^3$ function with open domain $\domain f \subset \mathbb{V}$, and which is $(\mathbb{H}^n_+, \beta)$-compatible with the domain $\closure\domain f$.
    \begin{enumerate}[label=(\roman*), ref=\ref{lem:compatibility-composition}(\roman*),leftmargin=25pt]
        \item \label{lem:compatibility-composition-i}Let $\mathcal{A}:\mathbb{H}^n\rightarrow\mathbb{H}^m$ be a positive linear map (i.e., $\mathcal{A}(\mathbb{H}^n_+)\subseteq\mathbb{H}^n_+$). Then $\mathcal{A}\circ f$ is $(\mathbb{H}^m_+, \beta)$-compatible with the domain $\closure\domain f$. 
        \item \label{lem:compatibility-composition-ii}Let $\mathcal{B}:\mathbb{V}'\rightarrow\mathbb{V}$ be an affine map satisfying $\image\mathcal{B}\cap\domain f\neq\varnothing$. Then $f\circ\mathcal{B}$ is $(\mathbb{H}^n_+, \beta)$-compatible with the domain $\mathcal{B}^{-1}(\closure\domain f) \coloneqq \{ x\in\mathbb{V}' : \mathcal{B}(x) \in \closure\domain f \}$.
    \end{enumerate}
\end{lem}

\section{Operator concavity and self-concordance}\label{sec:proof-main}

In this section, we present the main technical result of the paper which establishes a relationship between operator concavity along lines and compatibility. This is summarized in the following theorem. Note that this strengthens the statement in Theorem~\ref{thm:simplified} as, by considering functions which are not necessarily scalar-valued, it applies in greater generality and leads to a stronger conclusion. 

\begin{thm}\label{thm:compatibility-operator-concave}
    Let $\mathbb{V}$ and $\mathbb{V}'$ be finite-dimensional real vector spaces, let $\mathcal{K}\subset\mathbb{V}'$ be a proper, convex cone, and let the dual cone of $\mathcal{K}$ be denoted as $\mathcal{K}_*\subset\mathbb{V}'$. Let $f:\domain f\rightarrow\mathbb{V}'$ be a $C^3$ function with open domain $\domain f \subset \mathbb{V}$. Suppose that for all $z\in\mathcal{K}_*$, $x\in\domain f$ and $h\in\mathbb{V}$ such that $x\pm h \in \closure\domain f$, the function $F:(-1, 1)\rightarrow\mathbb{R}$ defined by
    \begin{equation*}
        F:(-1, 1)\rightarrow\mathbb{R}, \quad F(t) = \inp{z}{f(x + th)},
    \end{equation*}
    is operator concave on $(-1, 1)$. Then $f$ is $(\mathcal{K}, 1)$-compatible with respect to $\closure\domain f$.
\end{thm}
\begin{proof}
    We first show that $f$ is $\mathcal{K}$-concave. Since $F$ is operator concave, and therefore concave in the standard sense, it follows that $\inp{z}{\mathsf{D}^2f(x)[h, h]} = F''(0) \leq 0$. As this holds for all $z\in\mathcal{K}_*$, then $\mathsf{D}^2f(x)[h, h] \preceq_{\mathcal{K}} 0$ for all $x\in\domain f$ and $h\in\mathbb{V}$ such that $x\pm h \in \closure\domain f$. By scaling $h$ by an arbitrary constant, and using the fact that $\mathcal{K}$ is full-dimensional, we extend this result to hold for all $h\in\mathbb{V}$, which shows that $f$ is $\mathcal{K}$-concave, see, e.g.,~\cite[Lemma 5.1.2]{nesterov1994interior}. 
    
    We now prove the desired compatibility result. Given that $F$ is operator concave on $(-1, 1)$, we can use Lemma~\ref{lem:operator-convex} to show that $F$ has the following integral representation
    \begin{equation*}
        F(t) = F(0) + F'(0)t + \frac{1}{2}F''(0) \int_{-1}^1 \frac{t^2}{1-st} \, d\mu(s), \quad \forall t\in(-1, 1),
    \end{equation*}
    for a unique Borel probability measure $\mu$ on $[-1, 1]$. Let us define the integrand $\xi_s:(-1, 1)\rightarrow\mathbb{R}$ by
    \begin{equation*}
        \xi_s(t) = \frac{t^2}{1-st},
    \end{equation*}
    for $s\in[-1, 1]$. A straightforward computation shows that
    \begin{equation}\label{eqn:xi-derivatives}
        \xi_s'(t) = \frac{2t-st^2}{(1-st)^2}, \qquad \xi_s''(t) = \frac{2}{(1-st)^3}, \qquad \xi_s'''(t) = \frac{6s}{(1-st)^4},
    \end{equation}
    and therefore
    \begin{equation}\label{eqn:xi-compatibility}
        \xi_s'''(0) = 6s \geq -6 = -3\xi_s''(0),
    \end{equation}
     for all $s\in[-1, 1]$. Next, we recognize that
     \begin{align}\label{eqn:integral-derivatives}
         \inp{z}{\mathsf{D}^k f(x)[h, \ldots, h]} = F^{(k)}(0) &= \frac{1}{2} F''(0) \frac{d^k}{dt^k} \biggl|_{t=0}\,\int_{-1}^1 \xi_s(t) \, d\mu(s) \nonumber\\
         &= \frac{1}{2} F''(0) \int_{-1}^1 \xi_s^{(k)}(0) \, d\mu(s),
     \end{align}
     for $k=2$ and $k=3$. We can exchange the order of differentiation and integration using a similar argument to~\cite[Theorem A.3]{fawzi2023optimal}, as follows. Using the dominated convergence theorem (see, e.g.,~\cite[Corollary 5.9]{bartle2014elements}), it suffices to show for $k=1$, $2$ and $3$ that there are constants $\varepsilon>0$ and $C_k>0$ such that $\abs{\xi_s^{(k)}(t)} \leq C_k$ for all $(t, s)\in[-\varepsilon, \varepsilon]\times[-1, 1]$. Using the expressions~\eqref{eqn:xi-derivatives}, we can concretely set 
     \begin{equation*}
        C_1 = \frac{2\varepsilon+\varepsilon^2}{(1-\varepsilon)^2}, \qquad C_2=\frac{2}{(1-\varepsilon)^3}, \qquad C_3=\frac{6}{(1-\varepsilon)^4},
    \end{equation*}
     for any $0<\varepsilon<1$. This satisfies the assumptions required for us to apply the dominated convergence theorem, and thus we can repeatedly swap the order of differentiation and integration to obtain~\eqref{eqn:integral-derivatives}.
     
     Finally, noting that $F''(0)\leq0$ as $F$ is concave, we combine~\eqref{eqn:xi-compatibility} and~\eqref{eqn:integral-derivatives} to show that
     \begin{equation*}
         \inp{z}{\mathsf{D}^3f(x)[h, h, h]} \leq -3\inp{z}{\mathsf{D}^2f(x)[h, h]}.
     \end{equation*}
    As this holds for all $z\in\mathcal{K}_*$, it follows that
    \begin{equation*}
        \mathsf{D}^3f(x)[h, h, h] \preceq_\mathcal{K} -3 \mathsf{D}^2f(x)[h, h].
    \end{equation*}
     This is true for all $x\in\domain f$ and $h\in\mathbb{V}$ such that $x \pm h \in \closure\domain f$, which allows us to obtain the desired result by appealing to the definition of compatibility.
\end{proof}

Note that Theorem~\ref{thm:simplified}, which we introduced earlier in Section~\ref{sec:intro}, is a simple corollary of this result.

\begin{proof}[Proof of Theorem~\ref{thm:simplified}]
    This follows directly from Theorem~\ref{thm:compatibility-operator-concave} and Lemma~\ref{lem:compatibility-to-barrier}.
\end{proof}

\section{Proof of Theorems~\ref{thm:renyi-barrier} and~\ref{thm:direct-renyi-barrier}}\label{sec:main-proof}

To prove these results, we will make use of the following compatibility results. Item (i) of Lemma~\ref{lem:renyi-compatibility} is established in Section~\ref{sec:lem:renyi-compatibility-i}. Item (ii) of Lemma~\ref{lem:renyi-compatibility} is established in Section~\ref{sec:lem:renyi-compatibility-ii}. Item (iii) of Lemma~\ref{lem:renyi-compatibility} is established in Section~\ref{sec:lem:renyi-compatibility-iii}. The key to proving these compatibility results will be by appealing to Theorem~\ref{thm:compatibility-operator-concave}.

\begin{lem}\label{lem:renyi-compatibility}
    For any positive integer $n$, the following compatibility results hold:
    \begin{enumerate}[label=(\roman*), ref=\ref{lem:renyi-compatibility}(\roman*)]
        \item For $\alpha\in[\frac{1}{2}, 1]$, the function $\Psi_\alpha$ is $(\mathbb{R}_+, 1)$-compatible with the domain $\mathbb{H}^n_{+}\times\mathbb{H}^n_{+}$. \label{lem:renyi-compatibility-i}
        \item For $\alpha\in[1, 2]$, the function $-\Psi_\alpha$ is $(\mathbb{R}_+, 1)$-compatible with the domain $\mathbb{H}^n_{+}\times\mathbb{H}^n_{+}$. \label{lem:renyi-compatibility-ii}
        \item For $\alpha\in[\frac{1}{2}, 1)$, the function $-\bm{D}_\alpha$ is $(\mathbb{R}_+, 1)$-compatible with the domain $\mathbb{R}_+\times\mathbb{H}^n_{+}\times\mathbb{H}^n_{+}$. \label{lem:renyi-compatibility-iii}
    \end{enumerate}
\end{lem}

Self-concordance of the barrier functions in Theorems~\ref{thm:renyi-barrier} and~\ref{thm:direct-renyi-barrier} are then a direct consequence of Lemma~\ref{lem:compatibility-to-barrier} and Lemma~\ref{lem:renyi-compatibility}. Optimality of the barrier parameters in Theorem~\ref{thm:renyi-barrier} follow directly from~\cite[Corollary 3.13]{fawzi2023optimal}, while optimality of the barrier parameter in Theorem~\ref{thm:direct-renyi-barrier} is proven in Appendix~\ref{appdx:barrier-parameter}.

\subsection{Proof of Lemma~\ref{lem:renyi-compatibility-i}}\label{sec:lem:renyi-compatibility-i}

To prove $(\mathbb{R}_+, 1)$-compatibility of $\Psi_\alpha$ with the domain $\mathbb{H}^n_+\times\mathbb{H}^n_+$ for $\alpha\in[\frac{1}{2},1]$, we follow a similar complex analysis approach as~\cite{hiai2001concavity,hiai2013concavity,hiai2016concavity} which was used to prove concavity of $\Psi_\alpha$ for the same range of $\alpha$. By appealing to Theorem~\ref{thm:compatibility-operator-concave}, the desired compatibility result follows by showing that
\begin{align}\label{eqn:direction-of-psi}
    F(t) \coloneqq \Psi_\alpha(X+tH, Y+tV) = \tr\left[ \left((Y + tV)^\frac{1-\alpha}{2\alpha} (X + tH) (Y + tV)^\frac{1-\alpha}{2\alpha} \right)^\alpha \right],
\end{align}
is operator concave on the interval $(-1, 1)$ for all $X,Y\in\mathbb{H}^n_{++}$ and $H,V\in\mathbb{H}^n$ satisfying $X \pm H \succeq 0$ and $Y \pm V \succeq 0$. To do this, we first show that the transpose of $F$, i.e., $\hat{F}(t)\coloneqq t F(1/t)$, is operator monotone by using the following lemma. 

\begin{lem}\label{lem:f-is-pick}
    Let $X,Y\in\mathbb{H}^n_{++}$ and $H,V\in\mathbb{H}^n$ be any Hermitian matrices satisfying $X \pm H \succeq 0$ and $Y \pm V \succeq 0$, and let $\alpha\in[\frac{1}{2}, 1]$. Consider the function
    \begin{equation*}
        \hat{F}(t) \coloneqq \tr\left[ \left((tY + V)^\frac{1-\alpha}{2\alpha} (tX + H) (tY + V)^\frac{1-\alpha}{2\alpha} \right)^\alpha \right],
    \end{equation*}
    defined on the interval $(1, \infty)$. Then $\hat{F}$ has an analytic continuation to the upper half-plane $\mathbb{C}^+\coloneqq\{ z\in\mathbb{C} : \im z > 0 \}$ that maps $\mathbb{C}^+$ into $\mathbb{C}^+$.
\end{lem}
\begin{proof}
    This is an intermediate result in the proof of~\cite[Theorem 2.1]{hiai2013concavity}, where, in the notation of~\cite{hiai2013concavity}, $\Phi$ and $\Psi$ are the identity operators, $p=(1-\alpha)/\alpha$, $q=1$, and $s=\alpha$. 
\end{proof}

Given this result, Loewner's theorem implies that $\hat{F}$ is also operator monotone on $(1, \infty)$. We then use the following lemma to show that $F$ is operator concave on $(0, 1)$.

\begin{lem}\label{lem:transpose-is-concave}
    Let $f$ be an operator monotone function on $(\gamma, \infty)$ for some $\gamma>0$, and let $\hat{f}(x)=x f(1/x)$ be the transpose of $f$. Then $\hat{f}$ is operator concave on $(0, 1/\gamma)$.
\end{lem}
\begin{proof}
    See Appendix~\ref{appdx:proof-transpose-is-concave}.
\end{proof}

By swapping out $H, V$ for $-H, -V$, an identical argument shows that $F$ is also operator concave on $(-1, 0)$. Finally, we use the following lemma to extend operator concavity of $F$ to $t=0$. 

\begin{lem}\label{cor:localization}
    Let $a,b,c$ be real numbers satisfying $a<b<c$, and let $f:(a,c)\rightarrow\mathbb{R}$ be a $C^{2n}$ function. If $f$ is $\mathbb{H}^n_+$-convex on $(a, b)$ and $\mathbb{H}^n_+$-convex on $(b, c)$, then $f$ is also $\mathbb{H}^n_+$-convex $(a, c)$.
\end{lem}
\begin{proof}
    See Appendix~\ref{appdx:proof-localization}.
\end{proof}

As $F$ is $C^\infty$, we can apply the above lemma for all positive integers $n$, and conclude that $F$ is operator concave on the entire domain $(-1, 1)$. Using Theorem~\ref{thm:compatibility-operator-concave}, we conclude that $\Psi_\alpha$ is $(\mathbb{H}^n_+, 1)$-compatible with the domain $\mathbb{H}^n_+\times\mathbb{H}^n_+$, which concludes the proof.

\subsection{Proof of Lemma~\ref{lem:renyi-compatibility-ii}}\label{sec:lem:renyi-compatibility-ii}

We first focus on proving a more general result which we will then use to prove Lemma~\ref{lem:renyi-compatibility-ii}. We will use $\oplus$ to denote the direct sum of two matrices, and $\otimes$ to denote the Kronecker product of two matrices, i.e., for any two matrices $X\in\mathbb{C}^{n\times m}$ and $Y\in\mathbb{C}^{p\times q}$, we have
\begin{equation*}
    X\oplus Y = \begin{bmatrix}
        X & 0 \\ 0 & Y
    \end{bmatrix},
\end{equation*}
and
\begin{equation*}
    X\otimes Y = \begin{bmatrix}
        x_{11} Y & \ldots & x_{1m} Y \\
        \vdots & \ddots & \vdots \\
        x_{n1} Y & \ldots & x_{nm} Y
    \end{bmatrix},
\end{equation*}
where $x_{ij}$ denotes the $(i,j)$-th entry of $X$. An important property of the Kronecker product is that for any two Hermitian matrices $X,Y\in\mathbb{H}^n$ with diagonalizations $X=U_1\Lambda_1U_1^*$ and $Y=U_2\Lambda_2U_2^*$, where $U_1,U_2$ are unitary matrices and $\Lambda_1,\Lambda_2$ are real diagonal matrices, that $X\otimes Y$ is also Hermitian with the diagonalization
\begin{equation*}
    X\otimes Y = (U_1\otimes U_2)(\Lambda_1\otimes \Lambda_2)(U_1 \otimes U_2)^*.
\end{equation*}
Using this, it is also easy to verify for any $p\in\mathbb{R}$ that $(X \otimes Y)^p=X^p\otimes Y^p$. 

We are now ready to present the following general theorem showing that certain families of multivariate matrix concave functions are operator concave along lines, and hence are compatible with respect to their domains.

\begin{thm}\label{thm:nc-function}
    Let $f_n:(\mathbb{H}^n_{++})^d\rightarrow\mathbb{H}^n$ be a family of functions indexed by a positive integer $n$.
    Let $f_n$ satisfy the following two properties:
    \begin{itemize}
        \item Respects direct sums, i.e.,
        \begin{equation*}
            f_{n+m}(X_1\oplus Y_1, \ldots, X_d\oplus Y_d) = f_n(X_1, \ldots, X_d) \oplus f_m(Y_1, \ldots, Y_d),
        \end{equation*}
        for all $d$-tuples of $n\times n$ Hermitian matrices $X_i$ and $m\times m$ Hermitian matrices $Y_i$, and positive integers $n$ and $m$.
        \item Respects simultaneous unitaries, i.e.,
        \begin{equation*}
            f_n(UX_1 U^*, \ldots, UX_d U^*) = U f_n(X_1, \ldots, X_d) U^*,
        \end{equation*}
        for all $n\times n$ unitaries $U$, $d$-tuples of $n\times n$ Hermitian matrices $X_i$, and positive integers $n$.
    \end{itemize}
    If $f_n$ is $\mathbb{H}^n_+$-concave for all positive integers $n$, then $f_n$ is $(\mathbb{H}^n_+, 1)$-compatible with the domain $(\mathbb{H}^n_{++})^d$ for all positive integers $n$.
\end{thm}
\begin{rem}
    Functions satisfying assumptions like those in Theorem~\ref{thm:nc-function} arise naturally in non-commutative or free probability, operator theory, and systems theory. In the language, for instance, of~\cite{helton2011proper}, the disjoint union of all positive definite cones (of all sizes) is an example of a \emph{non-commutative domain}. Moreover, a family of functions defined on a non-commutative domain, that is compatible with direct sums and similarities, is a \emph{free mapping}.
\end{rem}
\begin{proof}
    By appealing to Theorem~\ref{thm:compatibility-operator-concave}, it suffices to show that
    \begin{equation*}
        F(t) \coloneqq \tr[A \cdot f_n(X_1+tH_1,\ldots,X_d+tH_d)],
    \end{equation*}
    is operator concave on $(-1, 1)$ for all $A\in\mathbb{H}^n_{+}$, $X_i\in\mathbb{H}^n_{++}$ and $H_i\in\mathbb{H}^n$ satisfying $X_i\pm H_i\succeq0$ for $i=1,\ldots,d$, and positive integers $n$. To prove this, we will show that the extension of $F$ to matrix arguments is $\mathbb{H}^m_{+}$-concave on the set of $m\times m$ Hermitian matrices with eigenvalues in $(-1, 1)$, for all positive integers $m$. We do this by showing that, for any matrix $T\in \mathbb{H}^m$ with eigenvalues in $(-1, 1)$, the matrix $F(T)$ can be expressed in the form
    \begin{equation}\label{eqn:composition}
        F(T) = \Xi(f_{nm}(\bm{X}_1(T), \ldots, \bm{X}_d(T))),
    \end{equation}
    where $\bm{X}_i(T)$ are affine maps for $i=1,\ldots,d$, and $\Xi$ is a positive linear map. The desired result then follows from $\mathbb{H}^{nm}_+$-concavity of $f_{nm}$. 
    
    We now establish that $F$ has a representation of the form~\eqref{eqn:composition}. Consider a matrix $T\in\mathbb{H}^m$ with eigenvalues in $(-1, 1)$ and diagonalization $T=U\Lambda U^*$ (where $U$ is unitary and $\Lambda$ is diagonal and real). Let us denote $\bm{X}_i(T)=\mathbb{I}\otimes X_i + T\otimes H_i$ and $X_i(\lambda)=X_i+\lambda H_i$ for $i=1,\ldots,d$. Using this notation, note that
    \begin{equation*}
        \bm{X}_i(T)=(U\otimes\mathbb{I})\bm{X}_i(\Lambda)(U^*\otimes\mathbb{I}),
    \end{equation*}
    and additionally that
    \begin{equation*}
        \bm{X}_i(\Lambda) = \begin{bmatrix} X_i(\lambda_1) & & \\
            & \ddots & \\
            & & X_i(\lambda_m)
        \end{bmatrix}.
    \end{equation*}
    Now using the two key properties of $f_n$ outlined in the statement of the theorem, we can show that
    \begin{align*}
        \MoveEqLeft[3] f_{nm}(\bm{X}_1(T), \ldots, \bm{X}_d(T)) \\
        &= (U\otimes\mathbb{I}) f_{nm}(\bm{X}_1(\Lambda), \ldots, \bm{X}_d(\Lambda)) (U^*\otimes\mathbb{I})\\
        &= (U\otimes\mathbb{I}) \begin{bmatrix}
            f_n(X_1(\lambda_1), \ldots, X_d(\lambda_1)) & & \\
            & \ddots & \\
            & & f_n(X_1(\lambda_m), \ldots, X_d(\lambda_m))
        \end{bmatrix} (U^*\otimes\mathbb{I}),
    \end{align*}
    where the first equality uses compatibility with simultaneous unitaries, and the second equality uses compatibility with direct sums. Finally, consider the linear map $\Xi:\mathbb{H}^{nm}\rightarrow\mathbb{H}^m$ defined by
    \begin{equation*}
        \Xi(M) = \sum_{i=1}^n \mu_i (\mathbb{I}\otimes v_i^*) M (\mathbb{I} \otimes v_i),
    \end{equation*}
    where $A$ has the eigendecomposition $A=\sum_{i=1}^n \mu_i v_iv_i^*$. This is the linear map which applies the operation $M_{ij}\mapsto\tr[AM_{ij}]$ to each $n\times n$ block of an $m\times m$ block matrix. Note that this is a positive linear map as $A\in\mathbb{H}^n_+$, which implies that $\mu_i\geq0$ for all $i=1,\ldots,n$. It is also relatively straightforward to confirm that
    \begin{equation*}
        \Xi((U\otimes\mathbb{I})M(U^*\otimes\mathbb{I}))=U\Xi(M)U^*,
    \end{equation*}
    for any unitary matrix $U$. Using these properties of $\Xi$, we can see that
    \begin{equation*}
        \Xi(f_{nm}(\bm{X}_1(T), \ldots, \bm{X}_d(T))) = U \begin{bmatrix}
            F(\lambda_1) & & \\
            & \ddots & \\
            & & F(\lambda_m)
        \end{bmatrix} U^* = UF(\Lambda)U^* = F(T).
    \end{equation*}
    Thus, we have shown that $T\mapsto F(T)$ is the composition between an $\mathbb{H}^{nm}_+$-concave function and a positive linear map from $\mathbb{H}^{nm}$ to $\mathbb{H}^{m}$, and therefore that $T\mapsto F(T)$ is $\mathbb{H}^m_+$-concave on $m\times m$ Hermitian matrices with eigenvalues in $(-1, 1)$. As this is true for any positive integer $m$, we conclude that $F$ is operator concave on $(-1, 1)$, as desired.
\end{proof}

Now to prove $(\mathbb{H}^n_+, 1)$-compatibility of $-\Psi_\alpha$ with the domain $\mathbb{H}^n_+\times\mathbb{H}^n_+$ for $\alpha\in[1,2]$, we use similar ideas as those used in~\cite{muller2013quantum,wilde2014strong} to prove joint convexity of $\Psi_\alpha$ for this range of $\alpha$. We recall that the noncommutative perspective of a function $g:(0, \infty)\rightarrow\mathbb{R}$ is defined as
\begin{equation*}
    P_g(X, Y) \coloneqq X^{\frac{1}{2}} g\!\left(X^{-\frac{1}{2}} Y X^{-\frac{1}{2}} \right) \!X^{\frac{1}{2}},
\end{equation*}
on the domain $\mathbb{H}^n_{++}\times\mathbb{H}^n_{++}$. When $g$ is operator concave, $P_g$ is jointly $\mathbb{H}^n_{+}$-concave for all positive integers $n$~\cite[Theorem 2.2]{ebadian2011perspectives}. Additionally, the noncommutative perspective satisfies the identity $P_g(X,Y)=P_{\hat{g}}(Y,X)$ for all $X,Y\in\mathbb{H}^n_{++}$, where $\hat{g}(x)=xg(1/x)$ is the transpose of $g$~\cite[Lemma 2.1]{hiai2017different}. We also introduce the following composition rule.

\begin{lem}\label{lem:perspective-composition}
    Consider the functions $g:(0,\infty)\rightarrow\mathbb{R}$ and $h:\domain h\rightarrow\mathbb{H}^n_{++}$ where $\domain h$ is a convex set. If $g$ is operator concave and operator monotone, and $h$ is $\mathbb{H}^n_+$-concave, then the function
    \begin{equation*}
        (X,y) \mapsto P_g(X, h(y)),
    \end{equation*}
    defined on $\mathbb{H}^n_{++}\times\domain h$ is jointly $\mathbb{H}^n_+$-concave.
\end{lem}
\begin{proof}
    See Appendix~\ref{appdx:proof-perspective-composition}.
\end{proof}

We are now ready to state a corollary of Theorem~\ref{thm:nc-function} which deals with these noncommutative perspectives, and which we will use to prove Lemma~\ref{lem:renyi-compatibility-ii}.

\begin{cor}\label{thm:compatibility-perspective}
    For the functions $g:(0,\infty)\rightarrow\mathbb{R}$ and $h:(0,\infty)\rightarrow(0,\infty)$, consider the composed noncommutative perspective function
    \begin{equation*}
        P_{g,h}(X,Y,Z) \coloneqq P_g(X, P_h(Y, Z)),
    \end{equation*}
    defined on the domain $\mathbb{H}^n_{++}\times\mathbb{H}^n_{++}\times\mathbb{H}^n_{++}$. If either 
    \begin{itemize}
        \item $g$ is operator concave, and $h$ is affine, or
        \item $g$ is operator concave and operator monotone, and $h$ is operator concave,
    \end{itemize}
     then $P_{g,h}$ is jointly $\mathbb{H}^n_+$-concave and $(\mathbb{H}^n_{+}, 1)$-compatible with the domain $\mathbb{H}^n_{+}\times\mathbb{H}^n_{+}\times\mathbb{H}^n_{+}$ for all positive integers $n$.
\end{cor}
\begin{proof}
    First, it is relatively straightforward to confirm that the noncommutative perspective, and therefore the composed noncommutative perspective $P_{g,h}$, respects direct sums and simultaneous unitaries as required in Theorem~\ref{thm:nc-function}. Therefore, by appealing to this theorem, it suffices to show that $P_{g,h}$ is jointly $\mathbb{H}^n_+$-concave for all positive integers $n$ under the given assumptions. If $h$ is affine, i.e., of the form $h(x)=ax+b$, then $P_{g,h}(X,Y,Z) = P_g(X, aY+bZ)$ and joint $\mathbb{H}^n_+$-concavity of $P_{g,h}$ follows immediately from joint $\mathbb{H}^n_+$-concavity of $P_g$ when $g$ is operator concave. When $h$ is operator concave and $g$ is both operator concave and operator monotone, $P_h$ is jointly $\mathbb{H}^n_{+}$-concave, and therefore joint concavity of $P_{g,h}$ follows directly from Lemma~\ref{lem:perspective-composition}. This concludes the proof.
\end{proof}

\begin{proof}[Proof of Lemma~\ref{lem:renyi-compatibility-ii}]
     Using Corollary~\ref{thm:compatibility-perspective} with $g(x)=-x^{1-\alpha}$ and $h(x)=x^{\frac{1}{\alpha}}$, we conclude that $P_{g,h}$ is $(\mathbb{H}^{n^2}_+, 1)$-compatible with the domain $\mathbb{H}^{n^2}_+\times\mathbb{H}^{n^2}_+\times\mathbb{H}^{n^2}_+$. By considering the positive linear map $\Phi:\mathbb{H}^{n^2}\rightarrow\mathbb{R}$ which satisfies $\Phi(X\otimes\conj{Y}) = \tr[XY]$ for all $X,Y\in\mathbb{H}^{n}$, a straightforward computation shows that
     \begin{align*}
         \Phi(P_{g,h}(X \otimes \mathbb{I}, Y\otimes\mathbb{I}, \mathbb{I} \otimes \conj{Y})) &= \Phi( P_{\hat{g}}(P_h(Y\otimes\mathbb{I}, \mathbb{I} \otimes \conj{Y}), X \otimes \mathbb{I}) ) \\
         &= \Phi\!\left(P_{\hat{g}}\!\left( 
         (Y^\frac{1}{2}\otimes\mathbb{I}) (Y^{-1}\otimes\conj{Y})^\frac{1}{\alpha} (Y^\frac{1}{2}\otimes\mathbb{I})
         , X\otimes\mathbb{I}\right)\right)  \\
         &= \Phi\!\left(P_{\hat{g}}\!\left(Y^{\frac{\alpha-1}{\alpha}}\otimes\conj{Y}^{\frac{1}{\alpha}}, X\otimes\mathbb{I}\right)\right)  \\
         &= -\Phi\!\left(  \left(Y^{\frac{\alpha-1}{2\alpha}}\otimes\conj{Y}^{\frac{1}{2\alpha}}\right) \left( Y^\frac{1-\alpha}{2\alpha} X Y^\frac{1-\alpha}{2\alpha} \otimes \conj{Y}^{-\frac{1}{\alpha}} \right)^\alpha \left(Y^{\frac{\alpha-1}{2\alpha}}\otimes\conj{Y}^{\frac{1}{2\alpha}}\right) \right)  \\
         &= -\Phi\!\left(Y^\frac{\alpha-1}{2\alpha} \left(Y^\frac{1-\alpha}{2\alpha} X Y^\frac{1-\alpha}{2\alpha} \right)^\alpha Y^\frac{\alpha-1}{2\alpha} \otimes \bar{Y}^\frac{1-\alpha}{\alpha} \right)\\
         &= -\Psi_\alpha(X, Y),
     \end{align*}
     where we have used the identity $P_g(X,Y)=P_{\hat{g}}(Y,X)$ and $\hat{g}(x)=xg(1/x)=-x^\alpha$ in the first line. Therefore, using Lemma~\ref{lem:compatibility-composition} gives the desired compatibility result.
\end{proof}

We can also recover all of the compatibility results in~\cite{fawzi2023optimal} by using Lemma~\ref{lem:compatibility-composition-ii} together with the identity $P_g(X, Y) = P_{g,h}(X, Y, Y)$ when $h(x)=x$. 

\begin{cor}\label{cor:prior-results}
    For an operator concave function $g:(0,\infty)\rightarrow\mathbb{R}$, the perspective function $P_g$ is $(\mathbb{H}^n_+, 1)$-compatible with the domain $\mathbb{H}^n_+\times\mathbb{H}^n_+$ for all positive integers $n$.
\end{cor}

\subsection{Proof of Lemma~\ref{lem:renyi-compatibility-iii}}\label{sec:lem:renyi-compatibility-iii}

The proof of $(\mathbb{R}_{+}, 1)$-compatibility of $\bm{D}_\alpha$ with respect to $\mathbb{R}_+\times\mathbb{H}^n_+\times\mathbb{H}^n_+$ for $\alpha\in[\frac{1}{2}, 1)$ is a simple extension of the proof for Lemma~\ref{lem:renyi-compatibility-i}. By appealing to Theorem~\ref{thm:compatibility-operator-concave}, it suffices to show that
\begin{align*}
    G(t)\coloneqq \bm{D}_\alpha\divz{u+ts}{X+tH}{Y+tV} = \frac{u+ts}{\alpha-1} \log \!\left( \Psi_\alpha \!\left(\frac{X+tH}{u+ts}, \frac{Y+tV}{u+ts} \right) \right),
\end{align*}
is operator convex on $(-1, 1)$ for all $u\in\mathbb{R}_{++}$, $X,Y\in\mathbb{H}^n_{++}$, $s\in\mathbb{R}$, and $H,V\in\mathbb{H}^n$ which satisfy $u\pm s \geq 0$, $X\pm H \succeq 0$, and $Y \pm V \succeq 0$. To prove this, we first use the fact that $\Psi_\alpha$ is homogeneous of degree one to show that
\begin{align*}
    G(t) 
    &= \frac{u+ts}{\alpha-1} \log \!\left( \frac{F(t)}{u+ts}  \right)
    = \frac{1}{\alpha-1}P_{\log}(u+ts, F(t)),
\end{align*}
where $F$ is defined in~\eqref{eqn:direction-of-psi}, and which we proved was operator concave on $(-1, 1)$ in Section~\ref{sec:lem:renyi-compatibility-i}. 
By applying Lemma~\ref{lem:perspective-composition} where $g(x)=\log(x)$ and $h(x)=F(x)$, it follows that the extension of $G$ to matrix arguments, i.e.,
\begin{equation*}
    G(T) = \frac{1}{\alpha-1}P_{\log}(u\mathbb{I}+sT, F(T))
\end{equation*}
where $T\in\mathbb{H}^m$ has eigenvalues in $(-1, 1)$, is $\mathbb{H}^m_+$-convex for any positive integer $m$. Therefore, we conclude that $G$ is operator convex on $(-1, 1)$. Appealing to Theorem~\ref{thm:compatibility-operator-concave} gives the desired compatibility result.

\begin{rem}
    By setting $u=1$ and $s=0$, a straightforward corollary of this result is that the (non-homogenized) negative sandwiched R\'enyi entropy $D_\alpha$ is also $(\mathbb{R}_+, 1)$-compatible with respect to the domain $\mathbb{H}^n_+\times\mathbb{H}^n_+$. This can be used to construct a $(1+2n)$-self-concordant barrier for $\closure\epigraph D_\alpha$ by appealing to Lemma~\ref{lem:compatibility-to-barrier}. However, as $D_\alpha$ is not homogeneous of degree one, this epigraph is not a cone, and optimality of the barrier parameter cannot be easily established using~\cite[Corollary 3.13]{fawzi2023optimal}. 
\end{rem}

\begin{rem}
    By combining the ideas discussed in this section, together with the ideas presented in the proofs of Corollary~\ref{thm:compatibility-perspective}, Corollary~\ref{cor:prior-results}, and~\cite[Corollary 1.8]{fawzi2023optimal}, we can also show that the (perspective) of the R\'enyi entropy $\hat{D}_\alpha$ is $(\mathbb{R}_+, 1)$-compatible with respect to its domain for $\alpha\in[0, 1)$. Therefore, suitable barrier functions for its epigraph can also be constructed by appealing to Lemma~\ref{lem:compatibility-to-barrier}.
\end{rem}

\section{Implementation}\label{sec:implementation}

We implement the barrier function for the epigraphs and hypographs proposed in Theorems~\ref{thm:renyi-barrier} and~\ref{thm:direct-renyi-barrier} in the primal-dual interior-point solver QICS~\cite{he2024qics}. These cones are accessible using the optimization modeling software PICOS~\cite{picos}. In Section~\ref{sec:derivatives}, we give some details about the derivative oracles required for these barriers. In Section~\ref{sec:experiments}, we present some numerical experiments to evaluate the performance of using the proposed barrier to solve optimization problems involving the sandwiched R\'enyi entropy. 

\subsection{Derivatives}\label{sec:derivatives}

We provide brief derivations of expressions for the first and second derivatives of the trace function 
\begin{equation}\label{eqn:general-trace-function}
    \Psi_{g,h}(X,Y) \coloneqq \tr\!\left[g\!\left(h(Y)^{\frac{1}{2}}Xh(Y)^{\frac{1}{2}}\right)\right],
\end{equation}
defined on $\mathbb{H}^n_{++}\times\mathbb{H}^n_{++}$, for twice continuously differentiable functions $g:(0,\infty)\rightarrow\mathbb{R}$ and $h:(0,\infty)\rightarrow\mathbb{R}$. Note that $\Psi_\alpha(X,Y)=\Psi_{g,h}(X,Y)$ when $g(x)=x^\alpha$ and $h(x)=x^{\frac{1-\alpha}{\alpha}}$. The derivatives of the barrier functions proposed in Theorem~\ref{thm:renyi-barrier} are then a straightforward consequence of these results. We refer the reader to~\cite[Section 4.1.1]{he2024qics} for a discussion of how to obtain derivatives for the barriers from the expressions presented in this section. 

First, by using the fact that $AB$ and $BA$ share the same nonzero eigenvalues for any $A\in\mathbb{C}^{n\times m}$ and $B\in\mathbb{C}^{m\times n}$ \cite[Theorem 1.32]{higham2008functions}, we can show that
\begin{equation*}
    \Psi_{g,h}(X,Y) = \tr\!\left[g\!\left(X^{\frac{1}{2}}h(Y)X^{\frac{1}{2}}\right)\right],
\end{equation*}
for all $X,Y\in\mathbb{H}^n_{++}$. Using this expression together with the original expression~\eqref{eqn:general-trace-function}, the first derivatives of $\Psi_{g,h}$ are relatively easy to derive using the chain rule~\cite[Theorem 3.4]{higham2008functions} combined with the derivative for trace functions, $\mathsf{D}(\text{tr}\circ g)(X)[H]=\inp{H}{g'(X)}$, see, e.g., \cite[Theorem 3.23]{hiai2014introduction}. Doing this gives
\begin{align*}
    \mathsf{D}_{X}\Psi_{g,h}(X, Y)[H] &= \inp*{h(Y)^{\frac{1}{2}}Hh(Y)^{\frac{1}{2}}}{\ g'\!\left(h(Y)^{\frac{1}{2}}Xh(Y)^{\frac{1}{2}} \right)} \\
    \mathsf{D}_{Y}\Psi_{g,h}(X, Y)[V] &= \inp*{X^{\frac{1}{2}}\mathsf{D}h(Y)[V]X^{\frac{1}{2}}}{\ g'\!\left(X^{\frac{1}{2}}h(Y)X^{\frac{1}{2}}\right)},
\end{align*}
where $X,Y\in\mathbb{H}^n_{++}$ and $H,V\in\mathbb{H}^n$. From these expressions, the second derivatives can be found by using the chain rule together with the product rule~\cite[Theorem 3.3]{higham2008functions} to obtain
\begin{align*}
    \mathsf{D}^2_{XX}\Psi_{g,h}(X, Y)[H_1, H_2] &= \inp*{h(Y)^{\frac{1}{2}}H_1h(Y)^{\frac{1}{2}}}{\ \mathsf{D}(g')\!\left(h(Y)^{\frac{1}{2}}Xh(Y)^{\frac{1}{2}}\right)\!\left[h(Y)^{\frac{1}{2}}H_2h(Y)^{\frac{1}{2}}\right]} \\
    \mathsf{D}^2_{YY}\Psi_{g,h}(X, Y)[V_1, V_2] &= \inp*{X^{\frac{1}{2}}\mathsf{D}h(Y)[V_1]X^{\frac{1}{2}}}{\ \mathsf{D}(g')\!\left(X^{\frac{1}{2}}h(Y)X^{\frac{1}{2}}\right)\!\left[X^{\frac{1}{2}}\mathsf{D}h(Y)[V_2]X^{\frac{1}{2}}\right]} \\
    & \hphantom{{}={}} +\inp*{X^{\frac{1}{2}}\mathsf{D}^2h(Y)[V_1, V_2]X^{\frac{1}{2}}}{\ g'\!\left(X^{\frac{1}{2}}h(Y)X^{\frac{1}{2}}\right)}.
\end{align*}
where $X,Y\in\mathbb{H}^n_{++}$ and $H_1,H_2,V_1,V_2\in\mathbb{H}^n$. Although the derivative $\mathsf{D}^2_{XY}\Psi_{g,h}(X, Y)[H, V]$ is not as straightforward to derive, we show how to do this using a similar method to the proof of~\cite[Lemma 4.2]{he2024qics}. 
\begin{lem}
    Consider the trace function $\Psi_{g,h}$ defined in~\eqref{eqn:general-trace-function}, where $g:(0,\infty)\rightarrow\mathbb{R}$ is twice continuously differentiable, and $h:(0,\infty)\rightarrow\mathbb{R}$ is continuously differentiable. If $\tilde{g}(x)=xg'(x)$, $X,Y\in\mathbb{H}^n_{++}$, and $H,V\in\mathbb{H}^n$, then
    \begin{equation*}
        \mathsf{D}^2_{XY}\Psi_{g,h}(X, Y)[H, V] = \inp*{\mathsf{D}h(Y)[V]}{\ h(Y)^{-\frac{1}{2}}\mathsf{D}\tilde{g}\!\left(h(Y)^{\frac{1}{2}}Xh(Y)^{\frac{1}{2}}\right)\!\left[h(Y)^{\frac{1}{2}}Hh(Y)^{\frac{1}{2}}\right] \!h(Y)^{-\frac{1}{2}}}.
    \end{equation*}
\end{lem}
\begin{proof}
    First, assume that $g'$ is a function of the form $g'(x)=x^p$ for any positive integer $p$. Then 
    \begin{align*}
        \mathsf{D}_{Y}\Psi_{g,h}(X, Y)[V] &= \inp*{\mathsf{D}h(Y)[V]}{\ X^{\frac{1}{2}}\!\left(X^{\frac{1}{2}}h(Y)X^{\frac{1}{2}}\right)^pX^{\frac{1}{2}}} \\
        &= \inp*{\mathsf{D}h(Y)[V]}{\ h(Y)^{-\frac{1}{2}}\!\left(h(Y)^{\frac{1}{2}}Xh(Y)^{\frac{1}{2}}\right)^{p+1}h(Y)^{-\frac{1}{2}}} \\
        &= \inp*{\mathsf{D}h(Y)[V]}{\ h(Y)^{-\frac{1}{2}}\tilde{g}\!\left(h(Y)^{\frac{1}{2}}Xh(Y)^{\frac{1}{2}}\right)\!h(Y)^{-\frac{1}{2}}}.
    \end{align*}
    Taking the derivative of this expression in the variable $X$ and the direction $H$ gives the desired result when $g'(x)=x^p$ for a positive integer $p$. By linearity, this is also true when $g$ is any polynomial function. This result is then extended to all twice continuously differentiable functions $g$ by using a similar continuity argument as~\cite[Theorem V.3.3]{bhatia2013matrix}. 
\end{proof}
The third derivatives of $\Psi_{g,h}$ can similarly be derived by using the chain and product rules on the second derivative expressions, which we omit the details of for berevity. Concrete expressions for the derivatives of the spectral functions $\mathsf{D}^k (g')(X)[H]$ and $\mathsf{D}^k h(Y)[V]$ can be found in, e.g.,~\cite[Theorem 3.33]{hiai2014introduction}.

\subsection{Numerical experiments}\label{sec:experiments}

Here, we provide numerical experiments using our implementation of our barrier functions in QICS to solve optimization problems involving the sandwiched R\'enyi entropy function. Note that in all experiments, we reformulate minimization of the sandwiched R\'enyi entropy as optimizing the sandwiched quasi-relative entropy $\Psi_\alpha$ by using~\eqref{eqn:reformulate}. In the following, we define the partial traces $\tr_1:\mathbb{H}^{n^2}\rightarrow\mathbb{H}^n$ and $\tr_2:\mathbb{H}^{n^2}\rightarrow\mathbb{H}^n$ to be the unique linear maps which satisfy
\begin{equation*}
    \tr_1(X \otimes Y)=\tr[X]Y, \qquad \quad \tr_2(X \otimes Y)=\tr[Y]X.
\end{equation*}
for all $X,Y\in\mathbb{H}^{n}$. Note that the partial traces are the adjoint operators of the linear maps $X\mapsto \mathbb{I}\otimes X$ and $X\mapsto X\otimes\mathbb{I}$, respectively, where we recall that $\otimes$ denotes the Kronecker product.

\paragraph{Sandwiched R\'enyi mutual information}

First, we consider computing the sandwiched R\'enyi mutual information, which is defined as the optimal value of
\begin{equation}\label{eqn:mutual-information}
    \min_{X\in\mathbb{H}^n}  \quad D_\alpha\divx{A}{\tr_2(A) \otimes X} \quad \subjto \quad \tr[X] = 1,\quad X \succeq0,
\end{equation}
for some matrix $A\in\mathbb{H}^{n^2}_{+}$. When $\alpha\rightarrow1$, it is known that the minimum is attained at $X_*=\tr_1(A)$. However, in general a closed-form expression for the sandwiched R\'enyi mutual information is not known. Instead, \cite[Lemma 5]{hayashi2016correlation} shows that the minimum is attained at an $X_*\in\mathbb{H}^n_+$ which is the unique fixed-point of the following map
\begin{equation}\label{eqn:mutual-information-residual}
    X_* = \tr_1(Z(X_*)) / \tr[Z(X_*)],
\end{equation}
where
\begin{equation*}
    Z(X_*) = \left( (\tr_2(A) \otimes X_*)^{\frac{1-\alpha}{2\alpha}} A (\tr_2(A) \otimes X_*)^{\frac{1-\alpha}{2\alpha}}\right)^\alpha.
\end{equation*}
We confirm this by solving~\eqref{eqn:mutual-information} using our implementation of the barrier functions in Theorem~\ref{thm:renyi-barrier} in QICS. The results from running these experiments are summarized in Table~\ref{tab:mutual-inf}.


\begin{table}
\caption{Results of computing the Sandwiched R\'enyi mutual information~\eqref{eqn:mutual-information} for randomly generated unit trace $n^2\times n^2$ Hermitian matrices $A$. The reported residual is the Frobenius norm of the residual matrix of~\eqref{eqn:mutual-information-residual}. Also shown is the code snippet used to solve~\eqref{eqn:mutual-information} using QICS and PICOS. Note that the constraint $X\succeq0$ is implied by the domain of \texttt{picos.sandquasientr}.}
\label{tab:mutual-inf}
\vspace{-2ex}\hrule\vspace{-7ex}
\begin{minipage}[c][1.01\width]{0.6\textwidth}%
\begin{center}
\begin{lstlisting}[language=python, numbers=none]
import numpy, picos

# Generate problem data
n, alpha = 4, 0.75
A = numpy.random.rand(n*n, 2*n*n).view(complex)
A = A @ A.conj().T

# Define problem
P = picos.Problem()
X = picos.HermitianVariable("X", n)
tr2_A = picos.partial_trace(A, 1, (n, n))

obj = picos.sandquasientr(A, tr2_A @ X, alpha)
P.set_objective("max", obj)
P.add_constraint(picos.trace(X) == 1)

# Solve problem
P.solve(solver="qics")
\end{lstlisting}
\par\end{center}%
\end{minipage}\centering{}%
\begin{minipage}[c][1.01\width]{0.4\textwidth}%
\begin{center}
    \small
    \begin{tabular}{@{}cccc@{}}
    \toprule
    $n$ & $\alpha$ & Time (s) & Residual \\ \midrule
    $4$ & $0.75$ & $0.09$ & $4.3\times10^{-8}$ \\
    $4$ & $1.50$ & $0.12$ & $3.5\times10^{-9}$ \\
    $8$ & $0.75$ & $3.40$ & $1.7\times10^{-7}$ \\
    $8$ & $1.50$ & $5.68$ & $8.3\times10^{-10}$ \\
    $16$  & $0.75$ & $693.50$ & $1.8\times10^{-8}$ \\
    $16$  & $1.50$ & $1557.06$ & $2.3\times10^{-10}$ \\ \bottomrule
    \end{tabular}
\par\end{center}%
\end{minipage}
\vspace{-9ex}\hrule
\end{table}

\paragraph{Quantum rate-distortion}

We next present some experiments to approximate the solution of quantum relative entropy programs by using the fact that the (normalized) sandwiched R\'enyi entropy converges to the quantum relative entropy as $\alpha\rightarrow1$. In particular, we estimate the optimal rate-distortion tradeoff for the maximally entangled state, which for a given constant $0\leq\delta\leq1$ is given as the solution to
\begin{equation}\label{eqn:qrd}
    \min_{X\in\mathbb{H}^{n^2}} \quad D_\alpha\divx{X}{\mathbb{I}\otimes\tr_1(X)} \quad \subjto \quad \tr_2(X) = \frac{1}{n}\mathbb{I}, \quad \delta \geq \inp{X}{\Delta}, \quad X\succeq0,
\end{equation}
as $\alpha\rightarrow1$, where
\begin{equation*}
    \Delta = \mathbb{I} - \frac{1}{n}\sum_{i=1}^n \sum_{j=1}^n e_ie_j^\top \otimes e_ie_j^\top,
\end{equation*}
and $\{e_i\in\mathbb{R}^n\}_{i=1}^n$ represent the standard basis. In~\cite[Theorem 4.16]{he2024efficient}, it was shown that the optimal value of this problem when $\alpha\rightarrow1$ is
\begin{equation}\label{eqn:qrd-closed-form}
    \log(n) + (1-\delta)\log(1-\delta) + \delta\log\!\left(\frac{\delta}{n^2-1}\right),
\end{equation}
whenever $0\leq\delta\leq1-1/n^2$, and is zero otherwise. We verify this by solving~\eqref{eqn:qrd} for values of $\alpha$ which converge to $1$. These problems are solved using our implementation of the barrier functions in Theorem~\ref{thm:renyi-barrier} in QICS, and we summarize these results in Table~\ref{tab:qrd}. Note that the sandwiched R\'enyi entropy is monotone in $\alpha$ for fixed matrix arguments~\cite[Theorem 7]{muller2013quantum}, which is reflected in these results.


\begin{table}
\caption{Results of estimating the optimal rate-distortion tradeoff of the maximally mixed state for $n=4$ and $\delta=0.25$ by solving~\eqref{eqn:qrd} for values of $\alpha$ converging to $1$. The optimal value reported for $\alpha=1$, shown in bold, is obtained from the closed form expression~\eqref{eqn:qrd-closed-form}. Also shown is the code snippet used to solve~\eqref{eqn:qrd} using QICS and PICOS. Note that the constraint $X\succeq0$ is implied by the domain of \texttt{picos.sandquasientr}.}
\label{tab:qrd}
\vspace{-2ex}\hrule\vspace{-10ex}
\begin{minipage}[c][1.022\width]{0.7\textwidth}%
\begin{center}
\begin{lstlisting}[language=python, numbers=none]
import numpy, picos

# Generate problem data
n, alpha, delta = 4, 0.99, 0.25
Delta = numpy.eye(n*n)
Delta[::n+1, ::n+1] -= 1/n

# Define problem
P = picos.Problem()
X = picos.SymmetricVariable("X", n*n)
tr1_X = picos.partial_trace(X, 0, (n, n))
tr2_X = picos.partial_trace(X, 1, (n, n))

obj = picos.sandquasientr(X, picos.I(n) @ tr1_X, alpha)
P.set_objective("max", obj)
P.add_constraint(tr2_X == picos.I(n) / n)
P.add_constraint(( X | Delta ) <= delta)

# Solve problem
P.solve(solver="qics")
\end{lstlisting}
\par\end{center}%
\end{minipage}\centering{}%
\begin{minipage}[c][1.022\width]{0.3\textwidth}%
\begin{center}
    \small
    \begin{tabular}{@{}lc@{}}
    \toprule
    $\alpha$ & Optimal value \\ \midrule
    $0.9$ & $0.0027555$ \\
    $0.99$ & $0.1332757$ \\
    $0.999$ & $0.1455750$ \\
    $0.9999$ & $0.1467922$ \\
    \bftab1 & \bftab0.1469467 \\
    $1.0001$ & $0.1470813$ \\
    $1.001$ & $0.1481874$ \\
    $1.01$ & $0.1604453$ \\
    $1.1$ & $0.2740472$ \\ \bottomrule
    \end{tabular}
\par\end{center}%
\end{minipage}
\vspace{-12ex}\hrule
\end{table}

\section{Conclusion}\label{sec:conclusion}

In this paper, we have established a close relationship between compatibility and operator concavity of a function, and used this result to prove the self-concordance of natural logarithmic barrier functions of the hypograph of $\Psi_\alpha$ for $\alpha\in[\frac{1}{2}, 1]$ and the epigraph of $\Psi_\alpha$ for $\alpha\in[1, 2]$. This allows us to solve optimization problems involving the sandwiched R\'enyi entropy using interior-point methods.

Whether there exists a self-concordant barrier for the epigraph of $\Psi_\alpha$ for $\alpha\in(2, \infty)$ with optimal barrier parameter $1+2n$ remains an open question. More generally, it is not clear how to construct an efficiently computable self-concordant barrier for this cone. For the scalar case, we can prove the following compatibility result.
\begin{prop}\label{prop:scalar-compatibility}
    For $\alpha\in[2,\infty)$, the function $(x,y)\mapsto-x^{\alpha}y^{1-\alpha}$ is $(\mathbb{R}_+, (2\alpha-1)/3)$-compatible with the domain $\mathbb{R}_+\times\mathbb{R}_+$.
\end{prop}
\begin{proof}
    See Appendix~\ref{appdx:proof-scalar-compatibilty}. 
\end{proof}
This compatibility parameter is tight in the sense that it cannot be improved, and so serves as a lower bound for the compatibility parameter for $\Psi_\alpha$ for $\alpha\in[2,\infty)$. Based on numerical experiments, we believe this lower bound is tight when $\Psi_\alpha$ is defined on matrices of any dimension, and therefore make the following conjecture.
\begin{conjecture}
    For $\alpha\in[2,\infty)$, the function $-\Psi_\alpha$ is $(\mathbb{H}^n_+, (2\alpha-1)/3)$-compatible with the domain $\mathbb{H}^n_+\times\mathbb{H}^n_+$ for any positive integer $n$.
\end{conjecture}
If this were true, then we could still use Lemma~\ref{lem:compatibility-to-barrier} to construct self-concordant barriers for the epigraph of $\Psi_\alpha$ for $\alpha\in[2,\infty)$, albeit with a barrier parameter which scales as $O(\alpha^3)$. As an alternative, it may be possible to find a lifted representation (see, e.g.,~\cite{fawzi2022lifting}) for the epigraph of $\Psi_\alpha$ which admits a barrier function with barrier parameter that does not grow with $\alpha$. 


Another possible extension of our results is to consider the general class of trace functions
\begin{equation*}
    \Psi_{p,q,s}(X, Y) = \tr\!\left[ \left(Y^\frac{q}{2} X^p Y^\frac{q}{2} \right)^s \right],
\end{equation*}
for $p\geq q$ and $s>0$. This generalizes the sandwiched R\'enyi entropy when $p=1$, $q=1-\alpha$, and $s=\alpha$, and the R\'enyi entropy when $p=\alpha$, $q=1-\alpha$, and $s=1$. The convexity properties of $\Psi_{p,q,s}$ for the full range of $p$, $q$ and $s$ are summarized in~\cite{zhang2020wigner}, which we repeat below for convenience. 
\begin{enumerate}[label=(\roman*)]
    \item $\Psi_{p,q,s}$ is jointly concave for $0\leq q\leq p \leq 1$ and $0<s\leq\frac{1}{p+q}$.
    \item $\Psi_{p,q,s}$ is jointly convex for $-1\leq q \leq p \leq 0$ and $s>0$.
    \item $\Psi_{p,q,s}$ is jointly convex for $-1\leq q \leq0$, $1\leq p < 2$, $(p,q)\neq (1, -1)$, and $s\geq\frac{1}{p+q}$.
\end{enumerate}
An interesting question is, to which range of parameters $p$, $q$ and $s$ do our compatibility results and techniques extend, thus allowing us to give self-concordant barriers for this larger class of functions. For scenario (i) and the subset of scenario (ii) where $-1\leq q \leq p \leq 0$ and $0<s\leq-\frac{1}{p+q}$, we can prove that $\Psi_{p,q,s}$ is $(\mathbb{R}_+, 1)$-compatible with respect to the domain $\mathbb{H}^n_+\times\mathbb{H}^n_+$ by using virtually the same proof as for Lemma~\ref{lem:compatibility-composition-i} presented in Section~\ref{sec:lem:renyi-compatibility-i}. The only modification required is to use the insights from the proof of~\cite[Theorem 2.1]{hiai2016concavity} (see, also,~\cite[Section 3]{carlen2018inequalities}) in place of Lemma~\ref{lem:f-is-pick}. 

It remains an open question how to construct self-concordant barriers for the remaining scenarios, i.e., the subset of (ii) where $-1\leq q \leq p \leq 0$ and $s>-\frac{1}{p+q}$, and scenario (iii). We note that scenario (iii) contains the $\alpha\in[1,\infty)$ range as a special case, and therefore a subset of this range may be amenable to a generalization of our approach in Section~\ref{sec:lem:renyi-compatibility-ii}.

\section*{Acknowledgments}

H. Fawzi was partially funded by UK Research and Innovation (UKRI) under the UK government’s Horizon Europe funding guarantee EP/X032051/1.

\appendix

\section{Auxiliary proofs}\label{appdx:proofs}

\subsection{Proof of optimality of barrier parameter in Theorem~\ref{thm:direct-renyi-barrier}}\label{appdx:barrier-parameter}

Here, we prove that any barrier for $\closure \epigraph \bm{D}_\alpha$ must have parameter at least $2+2n$ by using the same technique as the proof of~\cite[Corollary 3.13]{fawzi2023optimal}. To do this, we first introduce the following result.

\begin{lem}[{\cite[Proposition 3.11]{fawzi2023optimal}}]\label{lem:optimal-parameter}
    Let $n$ be a positive integer, and let $h:\mathbb{R}^n_{++}\rightarrow\mathbb{R}$ be convex and positively homogeneous of degree one. Then any self-concordant barrier for $\closure\epigraph h$ has barrier parameter at least $1+n$.
\end{lem}

Now consider the function $h:\mathbb{R}_{++}\times\mathbb{R}^n_{++}\times\mathbb{R}^n_{++}$ defined by 
\begin{equation*}
    h(u, x, y)=\bm{D}_\alpha\divz{u}{\diag(x)}{\diag(y)},
\end{equation*}
where $\diag(x)$ is the diagonal matrix with diagonal elements given by $x$. As $\bm{D}_\alpha$ is convex and positively homogeneous of degree one, so is $h$, and therefore Lemma~\ref{lem:optimal-parameter} tells us that any self-concordant barrier for $\closure\epigraph h$ has barrier parameter at least $2+2n$. This implies that any self-concordant barrier $(t,u,X,Y)\mapsto F(t,u,X,Y)$ for $\closure\epigraph h$ must also have parameter at least $2+2n$, as otherwise $(t,u,x,y)\mapsto F(t,u,\diag(x),\diag(y))$ would be a self-concordant barrier for $\closure\epigraph h$ with parameter less than $2+2n$, which completes the proof.

\subsection{Proof of Lemma~\ref{lem:transpose-is-concave}}\label{appdx:proof-transpose-is-concave}

As $f$ is operator monotone on $(\gamma, \infty)$, Loewner's theorem (see Lemma~\ref{lem:loewner}) tells us that $f$ has the integral representation
\begin{equation*}
    f(x) = \alpha + \beta x + \int_{-\infty}^\gamma \frac{1}{s-x}-\frac{s}{s^2+1}\,d\mu(s), \quad\forall x\in(\gamma,\infty),
\end{equation*}
for some constants $\alpha\in\mathbb{R}$ and $\beta\geq0$ and a positive Borel measure $\mu$ on $(-\infty, \gamma]$. The transpose of $f$ is therefore equal to
\begin{equation*}
    \hat{f}(x)=xf(1/x) = \alpha x + \beta + \int_{-\infty}^\gamma \frac{x^2}{sx-1}-\frac{sx}{s^2+1}\,d\mu(s), \quad\forall x\in(0, 1/\gamma).
\end{equation*}
It suffices to show that the integrand is operator concave on the domain $(0, 1/\gamma)$, for each $s\in(-\infty, \gamma]$. When $s=0$, the integrand is a negative quadratic, which is operator concave. When $s\neq0$, we have
\begin{equation*}
    \frac{x^2}{sx-1} = \frac{x}{s} + \frac{1}{s^2} + \frac{1}{s^2}\frac{1}{sx-1}.
\end{equation*}
Since $x\mapsto x^{-1}$ is operator convex on $(0, \infty)$, it follows that $h(x)\coloneqq(sx-1)^{-1}$ is operator concave on $(-\infty, 1/s)$ if $s\geq0$, and is operator concave on $(1/s, \infty)$ if $s\leq0$. Therefore, $h$ is operator concave on $(0,1/\gamma)$ for each $s\in(-\infty,0)\cup(0,\gamma]$, and therefore the entire integrand is operator concave on $(0,1/\gamma)$ for each $s\in(-\infty,\gamma]$, from which the desired result follows.

\subsection{Proof of Lemma~\ref{cor:localization}}\label{appdx:proof-localization}

We will use the following lemma to prove the desired result.

\begin{lem}[{\cite[Theorem 9.2]{simon2019loewner}}]\label{lem:hansen-tomiyama}
    Let $f$ be a $C^{2n}$ function on the interval $(a, b)$, and let $n$ be a fixed integer. Then $f$ is $\mathbb{H}^n_+$-convex if and only if the Hansen-Tomiyama matrix $H_n(x; f)$ with entries
    \begin{equation*}
        H_n(x; f)_{ij} = \frac{f^{(i+j)}(x)}{(i+j)!}, \quad 1 \leq i,j \leq n,
    \end{equation*}
    is positive semidefinite for all $x\in(a, b)$.
\end{lem}

Using Lemma~\ref{lem:hansen-tomiyama}, we know the Hansen-Tomiyama matrix $H_n(x; f)$ is positive semidefinite for all $x\in(a,b)\cup(b,c)$. As $f$ is $C^{2n}$, all entries of the Hansen-Tomiyama matrix are continuous on $(a, c)$, and therefore $H_n(b; f)$ must also be positive semidefinite. Appealing to Lemma~\ref{lem:hansen-tomiyama} again in the other direction gives the desired $\mathbb{H}^n_+$-convexity result.

\subsection{Proof of Corollary~\ref{lem:perspective-composition}}\label{appdx:proof-perspective-composition}

Let $X_1,X_2\in\mathbb{H}^n_{++}$ and $y_1,y_2\in\domain h$. For some $\lambda\in[0,1]$, let $X=\lambda X_1+(1-\lambda)X_2$ and $y=\lambda y_1+(1-\lambda)y_2$. As $h$ is $\mathbb{H}^n_+$-concave, we have by definition that
\begin{equation*}
    h(y) \succeq \lambda h(y_1) + (1-\lambda) h(y_2).
\end{equation*}
Using operator monotonicity of $g$, we can show that
\begin{align*}
    P_g(X, h(y)) &= X^{\frac{1}{2}} g\!\left(X^{-\frac{1}{2}} h(y) X^{-\frac{1}{2}}\right) \!X^{\frac{1}{2}} \\
    &\succeq X^{\frac{1}{2}} g\!\left(X^{-\frac{1}{2}} (\lambda h(y_1) + (1-\lambda) h(y_2)) X^{-\frac{1}{2}} \right)\! X^{\frac{1}{2}} \\
    &= P_{g}(X, \lambda h(y_1) + (1-\lambda) h(y_2) ).
\end{align*}
Finally, we use joint concavity of $P_g$ to show
\begin{align*}
    P_g(X, h(y)) &\succeq \lambda P_{g}(X_1, h(y_1) ) + (1-\lambda) P_{g}(X_2, h(y_2) ),
\end{align*}
which shows that the desired function is jointly $\mathbb{H}^n_+$-concave, as required.

\subsection{Proof of Proposition~\ref{prop:scalar-compatibility}}\label{appdx:proof-scalar-compatibilty}

First, it is straightforward to compute the second derivatives of $f$ as
\begin{equation*}
    \frac{\partial^2\! f}{\partial x^2} = -\alpha(\alpha-1)x^{\alpha-2}y^{1-\alpha}, \quad \frac{\partial^2\! f}{\partial x\partial y} = \alpha(\alpha-1)x^{\alpha-1}y^{-\alpha}, \quad \frac{\partial^2\! f}{\partial y^2} = -\alpha(\alpha-1)x^\alpha y^{-\alpha-1},
\end{equation*}
and the third derivatives of $f$ as
\begin{gather*}
    \frac{\partial^3\! f}{\partial x^3} = -\alpha(\alpha-1)(\alpha-2)x^{\alpha-3}y^{1-\alpha}, \quad \frac{\partial^3\! f}{\partial x^2\partial y} = \alpha(\alpha-1)^2x^{\alpha-2}y^{-\alpha}, \\
    \frac{\partial^3\! f}{\partial x\partial y^2} = -\alpha^2(\alpha-1)x^{\alpha-1}y^{-\alpha-1}, \quad \frac{\partial^3\! f}{\partial y^3} = \alpha(\alpha-1)(\alpha+1)x^{\alpha}y^{-\alpha-2},
\end{gather*}
Now consider $x,y>0$ and $h,v\in\mathbb{R}$ which satisfy $x\pm h \geq0$ and $y\pm v \geq0$, and let us denote $\hat{x}=h/x$ and $\hat{y}=v/y$, which satisfy $-1\leq\hat{x},\hat{y}\leq1$. The second and third directional derivatives can be expressed as
\begin{align*}
    \mathsf{D}^2f(x,y)[(h,v), (h,v)] &= -\alpha(\alpha-1)x^\alpha y^{1-\alpha}(\hat{x} - \hat{y})^2 \\
    \mathsf{D}^3f(x,y)[(h,v), (h,v), (h,v)] &= -\alpha(\alpha-1)x^\alpha y^{1-\alpha} (\hat{x}-\hat{y})^2((\alpha-2)\hat{x}-(\alpha+1)\hat{y}).
\end{align*}
Then using the fact that $\alpha>2$, we can show that 
\begin{equation*}
    (\alpha-2)\hat{x}-(\alpha+1)\hat{y} \leq 2\alpha-1.
\end{equation*}
Appealing to the definition of compatibility then gives the desired result.

\bibliographystyle{IEEEtran}
\bibliography{refs}

\end{document}